\documentclass[11pt, reqno]{amsart}
\usepackage[dvipsnames]{xcolor}
\usepackage[
    pagebackref,
    colorlinks=true,
    urlcolor=NavyBlue,
    linkcolor=NavyBlue,
    citecolor=NavyBlue
]{hyperref}
\usepackage{cleveref}
\usepackage{graphicx}
\usepackage{epsfig}
\usepackage[latin1]{inputenc}
\usepackage{amsmath}
\usepackage{amsfonts}
\usepackage{amssymb}
\usepackage{amsthm}
\usepackage{amscd}
\usepackage{verbatim}
\usepackage{subfigure}
\usepackage{caption}
\usepackage{pinlabel}
\usepackage{stmaryrd}
\usepackage{enumerate, enumitem}
\usepackage{todonotes}
\usepackage{bm}
\usepackage{thmtools}
\usepackage{thm-restate}
\usepackage{lipsum}
\usepackage{setspace}
\usepackage{mathtools}
\usepackage[all]{xypic}
\usepackage[abs]{overpic}
\usepackage[normalem]{ulem}
\usepackage[alphabetic,backrefs,msc-links]{amsrefs}

\allowdisplaybreaks

\usepackage{tikz}
\usetikzlibrary{arrows}
\usetikzlibrary{decorations.pathreplacing}
\usepackage{verbatim}
\usetikzlibrary{cd}
\tikzset{taar/.style={double, double equal sign distance, -implies}}
\tikzset{amar/.style={->, dotted}}
\tikzset{dmar/.style={->, dashed}}
\tikzset{aar/.style={->, very thick}}

\newtheorem{theorem}{Theorem}[section]

\newtheorem{lemma}[theorem]{Lemma}
\newtheorem{proposition}[theorem]{Proposition}
\newtheorem{conjecture}[theorem]{Conjecture}

\newtheorem{corollary}[theorem]{Corollary}

\newtheorem{question}[theorem]{Question}

\theoremstyle{definition}

\theoremstyle{remark}
\newtheorem{remark}[theorem]{Remark}

\newenvironment{reptheorem}[1]
{\begin{trivlist}
\item[\hskip\labelsep{\bfseries Theorem~\ref{#1}.}]\itshape}
{\end{trivlist}}

\newenvironment{repcorollary}[1]
{\begin{trivlist}
\item[\hskip\labelsep{\bfseries Corollary~\ref{#1}.}]\itshape}
{\end{trivlist}}

\def\F{\mathbb{F}}

\def\Z{\mathbb{Z}}

\def\d{\partial}

\def\CFK{\mathit{CFK}}

\newcommand\HFKm{\mathit{HFK}^-}
\newcommand\HFKhat{\widehat{\mathit{HFK}}}

\newcommand{\mh}{h_{\min}}
\newcommand{\hgt}{\operatorname{ht}}

\author[J.\ Hom]{Jennifer Hom}
\address {School of Mathematics, Georgia Institute of Technology, Atlanta, GA 30332}
\email{hom@math.gatech.edu}

\author[J.\ Park]{JungHwan Park}
\address{Department of Mathematical Sciences and RIM, Seoul National University, Seoul 08826}
\email{jungpark0817@snu.ac.kr}

\numberwithin{equation}{section}
\numberwithin{equation}{section}

\title{Ribbon concordance and cabling}

\begin{document}

\begin{abstract}
We study ribbon concordances to cable knots. We formulate a conjecture predicting that any nontrivial knot admitting a ribbon concordance to a $(p,q)$-cable must itself be a $(p,q)$-cable. We prove the conjecture when the knot admitting the ribbon concordance is already a cable of the same companion, is a torus knot, or has genus one. We also verify it for a broad class of target cables. The proofs use a minimum-height invariant defined from the immersed-curve formulation of knot Floer homology. This invariant obstructs ribbon concordances and implies that any nontrivial knot admitting a ribbon concordance to a fibered cable knot is prime. We also establish genus bounds for knots admitting ribbon concordances to cable knots.
\end{abstract}

\maketitle

\section{Introduction}

Knot concordance is an equivalence relation, but it admits a natural directed
refinement. Recall that two knots $J$ and $K$ in $S^3$ are
\emph{concordant} if they cobound a smoothly properly embedded annulus
\[
        C\subset S^3\times[0,1],
        \qquad
        \partial C=-J\times\{0\}\cup K\times\{1\}.
\]
Such a concordance is called a \emph{ribbon concordance} from $J$ to $K$
if the height function on $C$, induced by projection to $[0,1]$, has no local
maxima. In this case, we write
\[
        J\leq K.
\]
A knot is called \emph{slice} if it is concordant to the unknot, and
\emph{ribbon} if there exists a ribbon concordance from the unknot to it.

Since its introduction by Gordon~\cite{Gordon-ribbon}, ribbon concordance has been widely studied;
see, for example,~\cites{Gilmer,silver,Miyazaki1994, Zemke-ribbon,LevineZemke,Kang:2022,DLVW:2022, Agol, FMZ:2025, BaldwinSivek, BaldwinHanselmanSivek, AgolRen}. Recently, Agol~\cite[Theorem~1.1]{Agol} proved the striking result that the relation $\leq$ defines a partial order on the set of isotopy classes of knots, resolving a conjecture of Gordon~\cite[Conjecture~1.1]{Gordon-ribbon}. We refer to this
partial order as the \emph{ribbon concordance order}.

In this article, we initiate the study of ribbon concordance to cable knots. Recall that every nontrivial knot is either a torus knot, a hyperbolic knot, or a satellite knot~\cite{Thurston}. From the perspective of ribbon concordance, satellite knots lie between two familiar extremes. Torus knots are rigid: Gordon~\cite{Gordon-ribbon} proved that they are minimal with respect to the ribbon concordance order. Thus, if $J \leq T$ for a torus knot $T$, then $J=T$. Hyperbolic knots, by contrast, are flexible in this respect: Silver and Whitten~\cite{SilverWhitten} showed that, for every knot $J$, there exists a hyperbolic knot $H$ such that $J \leq H$; see also \cite{Myers}.

Satellite knots therefore provide the remaining natural setting in which to seek meaningful restrictions on predecessors under ribbon concordance. Although one can ask this question for arbitrary satellite knots, we focus here on one of the simplest satellite operations:~\emph{cabling}. Let $K_{p,q}$ denote the $(p,q)$-cable of $K$, where $p$ denotes the longitudinal winding number. Recall that $K_{p,q}$ is isotopic to $K_{-p,-q}$; hence, throughout the article, we assume without loss of generality that $p$ is positive.

\begin{conjecture}\label{conj:cabling}
Let $K$ be a knot, and let $p>1$. Suppose that $J$ is a nontrivial knot~with
\[
        J\leq K_{p,q}.
\]
Then $J$ is itself the $(p,q)$-cable of some knot.
\end{conjecture}

The aim of this article is to provide evidence for this conjecture using knot Floer homology. We begin with a complete answer to \Cref{conj:cabling} under the additional hypothesis that the predecessor is a cable of the same companion knot. This result shows that, within the family of cables of a fixed knot $K$, the ribbon concordance order is rigid:

\begin{theorem}\label{thm:ribboncable}
Let $K$ be a knot, and let $p>1$. Suppose that $J$ is a cable of $K$ with
\[
        J\leq K_{p,q}.
\]
Then $J$ is the $(p,q)$-cable of $K$.
\end{theorem}

We remark that, in the statement of \Cref{thm:ribboncable}, we impose no condition on the cabling parameters of $J$; in particular, $J$ may be the trivial cable $K$ itself. Thus \Cref{thm:ribboncable} implies, in particular, that the companion knot $K$ is not ribbon concordant to any nontrivial cable $K_{p,q}$. We also remark that, when $K$ is nontrivial, some of the most interesting consequences of \Cref{thm:ribboncable} are that $K_{p,1}$ is ribbon concordant to $K_{p',1}$ if and only if $p=p'$, and that $K_{q,p}$ is not ribbon concordant to $K_{p,q}$. For instance, if $K$ is slice, then all of the knots $K_{p,1}$ are concordant to one another, since they are all slice. Similarly, if $K$ is slice, then $K_{p,q}$ and $K_{q,p}$ are both concordant to the torus knot $T_{p,q}=T_{q,p}$. Thus \Cref{thm:ribboncable} detects distinctions among cables that are invisible to ordinary concordance.

Next, we turn to concrete examples satisfying \Cref{conj:cabling}. Many known cases of the conjecture follow from a stronger phenomenon: the target cable is minimal in the ribbon-concordance order\footnote{For instance, if $K$ is a $\gamma_0$-sharp fibered knot in the sense of Hom--Park, a class that includes tight fibered knots, then every cable $K_{p,q}$ is minimal in the ribbon-concordance order by~\cite{HomPark}; see also~\cite{Baker}. There are many other examples as well; for instance, the $(n,1)$-cables of the figure-eight knot are minimal~\cite{Gordon-ribbon,silver,miyazaki}.}. The family below is complementary: in these examples, \Cref{conj:cabling} is verified directly rather than as a formal consequence of minimality. 

Let $\|\cdot\|$ denote simplicial volume. By \cite[Corollary~4.2]{Gordon-satellite}, the knots $K$ with $\|S^3\smallsetminus K\|=0$ are precisely those in the smallest class containing the unknot and closed under connected sum and cabling; in particular, every such knot is fibered. We refer to $\|S^3\smallsetminus K\|$ as the \emph{simplicial volume} of $K$. The following theorem first establishes a knot Floer-theoretic primeness result: every nontrivial ribbon-concordance predecessor of a cable of a fibered knot is prime. When the companion has vanishing simplicial volume, Agol--Ren's monotonicity theorem \cite[Theorem~1.4]{AgolRen} and Gordon's characterization then imply that the predecessor is itself a cable.

\begin{theorem}\label{thm:fiberedprime}
Let $K$ be a fibered knot, and let $p>1$. Suppose that $J$ is a nontrivial knot with
\[
        J\leq K_{p,q}.
\]
Then $J$ is prime. Furthermore, if $\|S^3 \smallsetminus K\|=0$, then $J$ is
a cable of some knot.
\end{theorem}

The primeness conclusion in \Cref{thm:fiberedprime} depends essentially on the target being a cable knot; fiberedness of the target alone is not enough. By Miyazaki~\cite{Miyazaki:1998}, if a knot $K$ is obtained as a band sum of a split link $J_1\sqcup J_2$, then $J_1\#J_2\leq K$. Thus, a fibered knot obtained as a nontrivial band sum of two nontrivial knots may have a composite predecessor. For example, one may take $K=10_{99}$, with $J_1$ and $J_2$ the right- and left-handed trefoils, respectively.

With additional hypotheses, we can say even more: we can specify the cabling parameters, thereby verifying \Cref{conj:cabling}. In particular, the following corollary shows that, if $K$ is a slice knot with vanishing simplicial volume, then for any fixed $p>1$, \Cref{conj:cabling} holds for all sufficiently large $q$.

\begin{corollary}\label{cor:pqcable}
Let $K$ be a slice knot with $\|S^3 \smallsetminus K\|=0$, and let $p,q>1$ be relatively prime integers. Assume that either $q/p>g(K)$ or $\Delta_K(t^p)$ has no nonconstant divisor lying in $\Z[t^q]$. Suppose that $J$ is a nontrivial knot with
\[
        J \leq K_{p,q}.
\]
Then $J$ is the $(p,q)$-cable of some knot.
\end{corollary}

Here we give concrete examples satisfying the Alexander-polynomial hypothesis in \Cref{cor:pqcable}; see \Cref{rmk:moreexamples} for further examples. For a knot $K$ and a positive integer $n$, let $nK$ denote the $n$-fold connected sum of $K$ with itself, and let $-K$ denote the reverse of the mirror image of $K$. We write $T_{r,s}$ for the $(r,s)$-torus knot. Note that $n(T_{r,s}\# -T_{r,s})$ is ribbon; hence, in particular, its nontrivial cables are not minimal.

\begin{corollary}\label{cor:Tab}
Let $r$ and $s$ be distinct primes, and let $p,q>1$. Suppose that $J$ is a nontrivial knot with
\[
        J \leq \bigl(n(T_{r,s} \# -T_{r,s})\bigr)_{p,q}.
\]
Then $J$ is the $(p,q)$-cable of some knot.
\end{corollary}

For the proofs, we define a new numerical invariant $\mh(K)$, the \emph{minimum-height invariant}, which measures the minimum height of a homologically inessential component of the immersed curves associated to $K$; see Section~\ref{sec:immersed}. To be more precise, in the immersed-curve reformulation of knot Floer homology~\cite{HRW}, a knot $K$ determines a collection $\gamma(K)$ of decorated curves in a marked torus. This collection has a distinguished homologically essential component $\gamma_0(K)$, whose height, denoted by $h_0(K)$, is a nonnegative even integer; see \Cref{rmk:heightconstraints}. The invariant $\mh(K)$ records the minimum height of a homologically \emph{inessential} component of $\gamma(K)$. We use the convention that $\mh(K)=0$ if $\gamma(K)$ has no homologically inessential component, equivalently, if $\gamma(K)=\gamma_0(K)$. We compute $\mh(K_{p,q})$ in terms of $p$ and $\mh(K)$ in \Cref{prop:minheightcable}, and prove an inequality relating $\mh(J)$ and $\mh(K)$ when $J\leq K$ in \Cref{prop:minheightobstruction}; these results may be of independent interest.

Another key ingredient is the following primeness criterion:

\begin{theorem}\label{thm:mhprime}
Suppose that $K$ is a nontrivial fibered knot with $\mh(K)\neq 2$. Then $K$ is prime.
\end{theorem}

The proof of this criterion uses fiberedness differently from how it is usually used in ribbon concordance: the proof of \cite[Theorem~1.1]{BaldwinVV}, together with the relationship between $\HFKhat$ and $\HFKm$, shows that the minus version of knot Floer homology of a fibered knot contains an $\F$-summand. This naturally leads to the
following question:

\begin{question}\label{ques:Fsummand}
Does $\HFKm(K)$ contain an $\F$-summand for every nontrivial knot $K$?
\end{question}

An affirmative answer would allow the fiberedness hypothesis to be removed from Theorems~\ref{thm:fiberedprime} and~\ref{thm:mhprime}. We note that Question~\ref{ques:Fsummand} has an affirmative answer for every nontrivial knot in the class generated under connected sum~\cite{OS-knots} by fibered knots~\cite{BaldwinVV}, Floer-thin knots (including all alternating knots)~\cite{Petkova}, twisted Whitehead doubles~\cite{Hedden}, generalized Mazur satellites~\cite{Bodish}, and cables~\cite{HW-cables}.

Conjecture~\ref{conj:cabling} naturally leads to the following conjecture
concerning knot genus. Recall that if $J$ were the $(p,q)$-cable of a
knot $L$, then Schubert's genus formula~\cite{Schubert} would give
\[
        g(J)=p\cdot g(L)+g(T_{p,q}).
\]
Thus, if $J$ is a torus knot, then $L$ must be the unknot and
$J=T_{p,q}$. Otherwise, $L$ is nontrivial, and hence
$g(J)\geq p+g(T_{p,q})$.

\begin{conjecture}\label{conj:genus}
Let $K$ be a knot, and let $p>1$. Suppose that $J$ is a nontrivial knot
with
\[
        J\leq K_{p,q}.
\]
If $J$ is a torus knot, then $J=T_{p,q}$. Otherwise, $g(J)\geq p+g(T_{p,q})$.
\end{conjecture}

The next theorem provides evidence for \Cref{conj:genus}:

\begin{theorem}\label{thm:genus}
Let $K$ be a knot, and let $p>1$. Suppose that $J$ is a nontrivial knot
with
\[
        J\leq K_{p,q}.
\]
Then
\[
        g(J)\geq
        \left\lceil\frac{p\cdot h_0(K)}{2}\right\rceil+g(T_{p,q}).
\]
In particular, if $h_0(K)\neq 0$, then
$g(J)\geq p+g(T_{p,q})$. Moreover, if $J$ is a torus knot, then $J=T_{p,q}$.
\end{theorem}

When $J$ has genus one, we can say even more:

\begin{theorem}\label{thm:genus1}
Let $K$ be a knot, let $J$ be a genus-one knot, and let $p>1$.  Suppose that
\[
        J\leq K_{p,q}.
\]
Then
\[
        J=T_{p,q}
        \qquad\text{ and }\qquad
        \{p,|q|\}=\{2,3\}.
\]
In particular, $J$ is a trefoil knot, right-handed if $q>0$ and
left-handed if $q<0$.
\end{theorem}

These results are proved in Section~\ref{sec:fixed-cables-genus}. The genus bound follows from concordance invariance and the cabling formula for the height of the essential component. In the genus one case, the minimum-height obstruction eliminates all inessential components, after which the essential-height formula determines the cabling parameters and trefoil
detection identifies $J$.

\subsection*{Notation and conventions}
For a $(p,q)$-cable, we assume throughout that $p>0$. For a positive integer
$n$, we write $nK$ for the $n$-fold connected sum of $K$ with itself and
$-K$ for the reverse of the mirror image of $K$. Alexander polynomials are
normalized to have nonzero constant term.

\subsection*{Organization}

In Section~\ref{sec:immersed}, we recall the immersed-curve perspective on
knot Floer homology, define the height invariants used throughout the paper,
and establish their behavior under ribbon concordance and cabling.  In
Section~\ref{sec:knot-floer-complexes}, we record the algebraic knot Floer
input needed for the later arguments.  Section~\ref{sec:cable-detection}
proves the primeness and cable-detection results, including
Theorems~\ref{thm:fiberedprime} and~\ref{thm:mhprime}, and Corollaries~\ref{cor:pqcable}
and~\ref{cor:Tab}.  Section~\ref{sec:fixed-cables-genus} proves the results
on cables of a fixed knot and the genus bound, namely
Theorems~\ref{thm:ribboncable}, ~\ref{thm:genus}, and~\ref{thm:genus1}.  Finally,
Appendix~\ref{app:examples} records the ordinary concordance facts about
cables that are used in the proofs.

\subsection*{Acknowledgements}
The authors thank John Baldwin and Tye Lidman for helpful conversations. JH was partially supported by NSF grant DMS-2506400 and Georgia Tech's Elaine M. Hubbard Faculty Fellowship. JP was partially supported by the Samsung Science and Technology Foundation (SSTF-BA2102-02) and NRF grant RS-2025-00542968.

The authors used ChatGPT to help expand the examples in \Cref{cor:Tab}, specifically in the portion of the proof dealing with the cases $\ell=r$ and $\ell=s$. ChatGPT was also used to check the computation leading to the formula for $\widehat I_m(T_{p,q})$ in Appendix~\ref{app:examples}, which is derived by taking a limit of Borodzik's formula in~\cite[Section~3]{Borodzik-rho-iterated}. The authors verified all arguments and computations.

\section{Height invariants from immersed curves}\label{sec:immersed}

By \cite[Theorem 1.1]{HRW}, we have a knot invariant $\gamma(K)$ consisting of a collection of decorated curves $\gamma(K)=\{ \gamma_i \}$ in a marked torus. It is convenient to picture these curves in the plane, where we lift the marked point to the half-integer lattice (that is, lattice points of the form $n.5$ where $n$ is an integer). We think of the half-integer lattice as pegs, across which our immersed curves cannot pass. It follows from \cite{HRW} that when $K$ is a knot in $S^3$, there is always a homologically essential component of $\gamma(K)$, which we denote by $\gamma_0(K)$.

The \emph{height} of a component $\gamma_i$ is the vertical distance between the highest point and lowest point of $\gamma_i$, which we take to be integers. Recall that we can recover knot Floer homology from the immersed curve invariant by intersecting the immersed curve with a vertical line. If we take the intersection points to occur at integer heights, then the height corresponds to the Alexander grading.

We will be interested in the \emph{minimum height} of a homologically inessential component of $\gamma(K)$, namely,
\[
\mh(K) := \min \{ \hgt(\gamma_i) \mid i > 0 \},
\]
with the convention that $\mh(K)=0$ if $\gamma(K)=\gamma_0(K)$.
We also write
\[
        h_0(K)=\hgt(\gamma_0(K))
\]
for the height of the homologically essential component. Note that $h_0(K) \leq 2g(K)$ and that if $\gamma(K) = \gamma_0(K)$, then $h_0(K)=2g(K)$.

\begin{figure}[ht]
\centering
\labellist
\endlabellist
\includegraphics[scale=1]{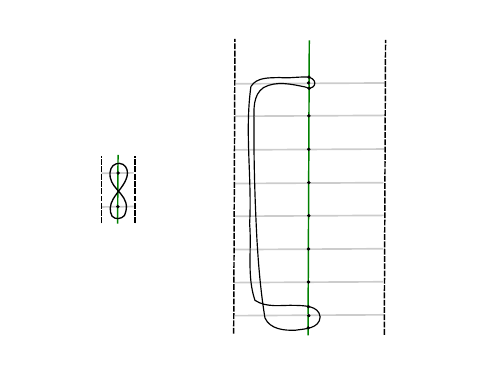}
\caption{Left, an immersed curve $\gamma_i$ with height 2. Right, an immersed curve of height 8.}
\label{fig:mh}
\end{figure}

\begin{remark}\label{rmk:heightconstraints}
By the symmetry of knot Floer homology, $h_0(K)$ is a nonnegative even integer. Moreover, if $\mh(K)\neq0$, then $\mh(K)\geq2$.
\end{remark}

We have that $\mh(K)$ gives a lower bound on $g(K)$:
\begin{lemma}\label{lem:mhgenus}
The invariant $\mh(K)$ gives a lower bound on (twice) the genus:
	\[ g(K) \geq  \left\lceil \frac{\mh(K)}{2}  \right\rceil. \]
\end{lemma}
\begin{proof}
The result follows from the fact that knot Floer homology detects genus \cite{OS-genus}, and the fact that if $\mh(K)=n$, then the knot Floer homology of $K$ is supported in Alexander gradings at least $n$ apart.
\end{proof}

The invariant $\mh$ is useful in obstructing ribbon concordances, as illustrated by the following proposition:

\begin{proposition}\label{prop:minheightobstruction}
If $J \leq K$ and $\mh(J)\neq 0$, then $\mh(J) \geq \mh(K)$.
\end{proposition}

\begin{proof}
The proposition follows from Zemke's proof that knot Floer homology obstructs ribbon concordance \cite[Theorem 1.7]{Zemke-ribbon} and the definition of $\mh$. That is, if $J \leq K$, then the knot Floer complex of $J$ injects into the knot Floer complex of $K$. In terms of immersed curves, this means that $\gamma(J)$ is a subset of $\gamma(K)$. We consider $\gamma_{>0}$, that is, we ignore the homologically essential component of $\gamma$. It follows that the minimum height component of $\gamma_{>0}(K)$ must be less than or equal to the minimum height component of $\gamma_{>0}(J)$, as desired.
\end{proof}

Recall that $\gamma_0$ is invariant under ordinary concordance~\cite[Proposition~2]{HW-cables} (see also \cite[Theorem~1]{Hom-conc}); hence, we immediately have the following.

\begin{lemma}\label{lem:essential-ribbon}
If $J\leq K$, then $h_0(J)=h_0(K)$. \qed
\end{lemma}

\begin{proposition}\label{prop:minheightcable}
Suppose that $\mh(K)\neq 0$ and let $p$ be a positive integer. Then
\[ \mh(K_{p,q}) = p \cdot \mh(K) -p+1. \]
\end{proposition}

\begin{proof}
We consider the cabling formula for immersed curves from \cite{HW-cables}. We are only interested in $\gamma_{>0}$. The recipe for $(p,q)$-cabling is as follows: each component $\gamma_i$ of $\gamma_{>0}$ spawns $p$ components. These $p$ components are obtained by 
\begin{enumerate}
	\item taking $p$ copies of $\gamma_i$, each stretched by a factor of $p$, 
	\item shifting each subsequent copy down from the preceding copy by $q$ units, and
	\item compressing the copies horizontally.
\end{enumerate}	
	See Figure \ref{fig:cables}. It follows from this algorithm that an immersed curve $\gamma_i$ of height $h$ becomes an immersed curve of height $p h -p +1$. Indeed, the difference between the highest and lowest pegs enclosed by $\gamma_i$ is $h-1$, and so when we stretch by a factor of $p$, this difference becomes $p(h-1)$. Hence the resulting height is $p(h-1)+1$, as desired.
\end{proof}

\begin{proposition}\label{prop:essentialheightcable}
For every knot $K$ and every cable $K_{p,q}$,
\[
        h_0(K_{p,q})
        =
        p\cdot h_0(K)+(p-1)(|q|-1).
\]
\end{proposition}

\begin{proof}
This is the cabling formula \cite{HW-cables} applied to the homologically essential component. 
We first consider the case $q>0$. We would like to know the height of the homologically essential component. Recall that we have $p$ copies of $\gamma_0$, each stretched by a factor of $p$, and shifted down $q$ units. The difference between the top Alexander grading and the bottom Alexander grading of each stretched copy of $\gamma_0$ is $p\cdot h_0(K)-(p-1)$. We have $p-1$ subsequent copies of the stretched $\gamma_0$, each shifted down by $q$ units, for a total of $(p-1)q$. Thus, the height is
\[ p\cdot h_0(K) -(p-1) +(p-1)q,\]
as desired.
When $q<0$, a similar argument yields a height of
\[ p\cdot h_0(K) -(p-1) +(p-1)(-q).\]
This completes the proof.
\end{proof}

\begin{figure}[ht]
\centering
\labellist
\endlabellist
\includegraphics[scale=1]{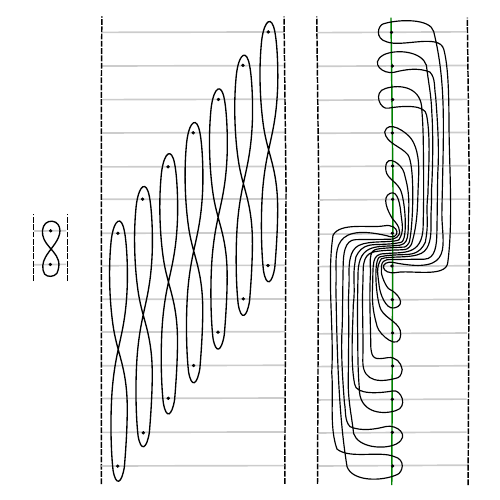}
\caption{The $(7, -1)$-cable of a height two component $\gamma_i$.}
\label{fig:cables}
\end{figure}

%%%%%%%%%%%%%%%%%
%%%%%%%%%%%%%%%%%
%%%%%%%%%%%%%%%%%
%%%%%%%%%%%%%%%%%
%%%%%%%%%%%%%%%%%

\section{Knot Floer complexes}\label{sec:knot-floer-complexes}
For our discussion of knot Floer complexes \cite{OS-knots, Rasmussen}, we use the notation of \cite{Hom-lecture}, where $\CFK(K)$ is viewed as bigraded chain complex over $\F[U,V]$. Then 
\[ \HFKm(K) = H_*(\CFK(K)/ V=0). \]
The immersed curve invariant $\gamma(K)$ contains the exact same information as $\CFK(K)/(UV=0)$. See \cite[Section 2]{HLP-unknotting} and \cite[Section 2]{HomPark} for further exposition on the relationship between $\CFK(K)$ and $\gamma(K)$.

\begin{proposition}\label{prop:Fsummand}
Suppose that $\HFKm(J_1)$ and $\HFKm(J_2)$ each have an $\F$-summand. Then 
\[ \mh(J_1 \# J_2) = 2. \]
\end{proposition}

\begin{proof}
By the hypothesis that $\HFKm(J_1)$ has an $\F$-summand, there exists a basis $\{x_i\}$ for $\CFK(J_1)$ with
\[ \d_U(x_1) = U x_2. \]
By the hypothesis that $\HFKm(J_2)$ has an $\F$-summand together with the symmetry of $\HFKm$, there exists a basis $\{y_i\}$ for $\CFK(J_2)$ with
\[ \d_V(y_1) = V y_2. \]

The main idea of the proof is that we will use the two pairs of generators $\{x_1, x_2\}$ and $\{y_1, y_2\}$ to form a $1 \times 1$ box (roughly generated by $\{x_1y_1, x_1y_2, x_2y_1, x_2y_2\}$) in $\CFK(J_1 \# J_2) = \CFK(J_1) \otimes_{\F[U,V]} \CFK(J_2)$, giving us the desired height 2 summand.

We will show that 
\begin{align*} 
a &= x_1y_1 \\
b &= x_1y_2 + V^{-1} (\d_V x_1) y_1 \\
c &= x_2y_1 + U^{-1} x_1 \d_U y_1 \\
d &= x_2y_2 + U^{-1} x_1 \d_U y_2 + V^{-1} (\d_V x_2) y_1 + U^{-1}V^{-1} (\d_V x_1) \d_U y_1
\end{align*}
form a direct summand of $\CFK(J_1 \# J_2)$ of height 2.

Indeed, it is straightforward to verify that
\begin{align*}
\d_V a &= V b \\
\d_U a &= U c \\
\d_U b &= U d \\
\d_V c &= V d,
\end{align*}
or, more visually,

\[\begin{tikzcd}
	c & {} & a \\
	\\
	d && b.
	\arrow["V", from=1-1, to=3-1]
	\arrow["U"', from=1-3, to=1-1]
	\arrow["V", from=1-3, to=3-3]
	\arrow["U"', from=3-3, to=3-1]
\end{tikzcd}\]

However, there may be other arrows coming into $a,b,c,$ and $d$. We now perform a change of basis to split off $\{a, b, c, d \}$ as a direct summand. If there is an incoming arrow to $d$, that is, $U^m V^n d$ appears as a term in $\d z$ for some basis element $z$ (where $z \neq b, c$), where at least one of $m$ and $n$ is positive. Suppose $m$ is positive (the case $n$ positive is analogous). Replacing $z$ with $z' = z + U^{m-1}V^n b$ eliminates the arrow from $z$ to $d$, and does not introduce any new arrows into $d$. Continuing in this manner, we can remove all incoming arrows to $d$ (besides the arrows from $b$ and $c$).

We now consider arrows into $b$. Suppose there is an arrow from some basis element $z$ (where $z \neq a$) to $b$. That is, $U^m V^n b$ appears as a term in $\d z$. Furthermore, using $\d^2=0$, we see that $n \geq 1$. 
%Then the path $z \to b \to d$ contributes $U^{m+1} V^n d$ to $\d^2 z$. Since $\d^2=0$, there must be another path from $z$ to $d$ to cancel that contribution (mod $2$). The only other incoming arrow to $d$ is from $c$. Hence $U^{m+1} V^{n-1} c$ must also appear in $\d z$. 
Replace $z$ with $z' = z + U^m V^{n-1} a$, which eliminates the arrow from $z$ to $b$, and does not introduce any new arrows to $b$ or $d$. Continuing in this manner, we can remove all incoming arrows to $b$ (besides the arrow from $a$ to $b$).

Now $\d^2=0$ implies that there are not any incoming arrows to $c$ (besides the arrow from $a$ to $c$), and similarly, $\d^2=0$ also implies that there are not any incoming arrows to $a$. Thus, $\{a, b, c, d \}$ forms a direct summand, as desired.
\end{proof}

Finally, we recall some basic symmetries of knot Floer homology $\HFKhat$
and a cabling result for $\HFKhat$. Recall that
\[
        \HFKhat(K)=H_*(\CFK(K)/(U=0,V=0)).
\] We write $\HFKhat_m(K,A)$ for the summand of $\HFKhat(K)$ in Maslov grading $m$ and Alexander grading $A$. Here, we summarize the known symmetries on knot Floer homology $\HFKhat$ from~\cite{OS-knots} (see also \cite{Hedden-Watson:2018}):

\begin{proposition}[{\cite[Section 3]{OS-knots}}]\label{prop:symmetries}
For every knot $K \subset S^3$ and for all $m,A \in \mathbb{Z}$, we have
\begin{equation*}\label{eq:HFK-alexander-symmetry}
        \HFKhat_m(K,A) \cong \HFKhat_{m-2A}(K,-A)
\end{equation*}
and
\begin{equation*}\label{eq:HFK-mirror-symmetry}
        \HFKhat_m(-K,A) \cong \HFKhat_{-m}(K,-A),
\end{equation*}
where $-K$ denotes the reverse of the mirror image of $K$.
\end{proposition}

Finally, Cheng--Hedden--Sarkar~\cite{Cheng-Hedden-Sarkar:2025} compute knot Floer homology in the top Alexander grading; for our purposes, we only need the following $(p,1)$-cable case.

\begin{proposition}[{\cite[Corollary 1.7]{Cheng-Hedden-Sarkar:2025}}]\label{prop:cablingHFKhat}
Let $K$ be a knot of genus $g$, and let $p>1$. Then
$g(K_{p,1})=pg$ and there is an isomorphism
\begin{equation*}\label{eq:HFK-cable-top-Alexander}
        \HFKhat_m(K_{p,1},pg) \cong \HFKhat_m(K,g)
\end{equation*}
for each $m\in \Z$.
\end{proposition}

\section{Primeness and cable detection}\label{sec:cable-detection}
We now apply Proposition~\ref{prop:Fsummand} to prove the results on prime and
cable predecessors stated in the introduction. We begin with
Theorem~\ref{thm:fiberedprime}, whose statement we recall.

\begin{reptheorem}{thm:fiberedprime}
Let $K$ be a fibered knot, and let $p>1$. Suppose that $J$ is a nontrivial knot with
\[
        J\leq K_{p,q}.
\]
Then $J$ is prime. Furthermore, if $\|S^3 \smallsetminus K\|=0$, then $J$ is
a cable of some knot.
\end{reptheorem}

\begin{proof}
Suppose that $K$ is fibered. Then $K_{p,q}$ is fibered as well~\cite{Stallings}. If $J\leq K_{p,q}$, then Miyazaki's theorem~\cite{miyazaki}, building on work of Silver~\cite{silver} and Kochloukova~\cite{Kochloukova}, implies that $J$ is fibered.

Suppose that $J$ is composite, say $J=J_1\# J_2$. If $J$ is fibered, then both $J_1$ and $J_2$ are fibered~\cite{Gabai}. The proof of~\cite[Theorem~1.1]{BaldwinVV}, together with the relationship between $\HFKhat$ and $\HFKm$, implies that $\HFKm$ of a fibered knot contains an $\mathbb{F}$-summand. Then Proposition~\ref{prop:Fsummand} implies that $\mh(J)=2$. In particular, $\gamma(J)$ has a homologically inessential component, which implies that both $\gamma(K_{p,q})$ and $\gamma(K)$ also have a homologically inessential component~\cite{Zemke-ribbon, HW-cables}. 
Propositions~\ref{prop:minheightobstruction} and
\ref{prop:minheightcable}, together with Remark~\ref{rmk:heightconstraints}, imply that
\[
        \mh(J)\geq\mh(K_{p,q})\geq3,
\]
contradicting $\mh(J)=2$. Hence $J$ must be prime.

Now suppose furthermore that $\|S^3\smallsetminus K\|=0$. By \cite[Theorem~1.4]{AgolRen}, we have $$\|S^3\smallsetminus J\|\leq \|S^3\smallsetminus K_{p,q}\|=0.$$ Hence, by \cite[Corollary~4.2]{Gordon-satellite}, the knot $J$ lies in the smallest class containing the unknot and closed under connected sum and cabling. Since $J$ is nontrivial and prime, it must in fact be a cable, as desired.
\end{proof}

\begin{remark}
The proof of Theorem~\ref{thm:fiberedprime} applies more generally: if $J$ is a nontrivial fibered knot and $J\leq K_{p,q}$ for some knot $K$ and $p>1$,
then $J$ is prime. If, in addition,
$\|S^3\smallsetminus J\|=0$, then $J$ is a cable of some knot.
\end{remark}

Theorem \ref{thm:mhprime} also follows readily, whose statement we recall:

\begin{reptheorem}{thm:mhprime}
Suppose that $K$ is a nontrivial fibered knot with $\mh(K)\neq2$. Then
$K$ is prime.
\end{reptheorem}

\begin{proof}
We prove the contrapositive. Suppose that $K$ is fibered and composite. Then $K=K_1\#K_2$, where $K_1$ and $K_2$ are fibered~\cite{Gabai}. As in the proof of \Cref{thm:fiberedprime}, the proof of \cite[Theorem~1.1]{BaldwinVV}, together with the relationship between $\HFKhat$ and $\HFKm$, implies that $\HFKm$ of a fibered knot contains an $\F$-summand. Proposition~\ref{prop:Fsummand} therefore implies that $\mh(K)=2$.
\end{proof}

The rest of this section is devoted to providing nontrivial examples where
\Cref{conj:cabling} is satisfied. We first give a criterion:

\begin{repcorollary}{cor:pqcable}
Let $K$ be a slice knot with $\|S^3 \smallsetminus K\|=0$, and let
$p,q>1$ be relatively prime integers. Suppose that either
$q/p>g(K)$
or $\Delta_K(t^p)$ has no nonconstant divisor lying in $\Z[t^q]$. Suppose that $J$ is a
nontrivial knot with
\[
        J \leq K_{p,q}.
\]
Then $J$ is the $(p,q)$-cable of some knot.
\end{repcorollary}

\begin{proof} Suppose that $J \leq K_{p,q}$. By \Cref{thm:fiberedprime}, $J$ is a cable. Write
\[
        J=L_{p',q'}
\]
with $p'>1$. Applying \Cref{prop:slice-full-LT}, we have
\[
        \{p,q\}=\{p',q'\}.
\]
Thus either $J=L_{p,q}$, in which case we are done, or
\[
        J=L_{q,p}.
\]
We assume the latter and in fact prove that $L$ is the unknot, which implies that $$J=T_{q,p}=T_{p,q},$$
which is a $(p,q)$-cable.

Assume first that $q/p>g(K)$. Since genus is monotone under ribbon
concordance,
\[
        g(L_{q,p})\leq g(K_{p,q}).
\]
Schubert's genus formula~\cite{Schubert} gives
\[
        q\cdot g(L)+\frac{(q-1)(p-1)}{2}
        \leq
        p\cdot g(K)+\frac{(p-1)(q-1)}{2}.
\]
Hence $qg(L)\leq pg(K)$, and combined with $q/p>g(K)$, this implies that
$g(L)=0$, that is, $L$ is the unknot.

Now assume instead that $\Delta_K(t^p)$ has no nonconstant divisor lying
in $\Z[t^q]$. By Gilmer~\cite{Gilmer},
\[
        \Delta_J(t) \mid \Delta_{K_{p,q}}(t).
\]
Using the cabling formula for Alexander polynomials, we have
\[
        \Delta_J(t)
        =
        \Delta_L(t^q)\Delta_{T_{q,p}}(t)
        =
        \Delta_L(t^q)\Delta_{T_{p,q}}(t),
\]
whereas
\[
        \Delta_{K_{p,q}}(t)
        =
        \Delta_K(t^p)\Delta_{T_{p,q}}(t).
\]
In particular, we have
\[
        \Delta_L(t^q)\mid \Delta_K(t^p).
\]
Since $\Delta_L(t^q)\in \Z[t^q]$, the hypothesis forces
$\Delta_L(t)=1$. Finally, Agol--Ren's monotonicity theorem~\cite{AgolRen} for simplicial
volume under ribbon concordance implies that $\|S^3\smallsetminus L\|=0$,
and combining this with Gordon's characterization of knots with zero
simplicial volume~\cite[Corollary~4.2]{Gordon-satellite} and
$\Delta_L(t)=1$ implies that $L$ is the unknot.
\end{proof}

We are now ready to give concrete examples satisfying \Cref{conj:cabling}:
\begin{repcorollary}{cor:Tab}Let $r$ and $s$ be distinct primes, and let $p,q>1$. Suppose that $J$ is a nontrivial knot with
\[
        J \leq \bigl(n(T_{r,s} \# -T_{r,s})\bigr)_{p,q}.
\]
Then $J$ is the $(p,q)$-cable of some knot.
\end{repcorollary}

\begin{proof} Let $K=n(T_{r,s} \# -T_{r,s})$. 
Since $r$ and $s$ are distinct primes,
\[
        \Delta_{T_{r,s}}(t)=\Phi_{rs}(t),
\]
so the roots of $\Delta_K(t)$ are primitive $rs$th roots of unity.
We will prove that $\Delta_K(t^p)$ has no nonconstant divisor lying
in $\Z[t^q]$, and then apply \Cref{cor:pqcable}.

Suppose, toward a contradiction, that $A(t^q)$ is a nonconstant divisor of
$\Delta_K(t^p)$, and let $\lambda$ be a root of $A(t^q)$. Then
$\zeta\lambda$ is also a root of $A(t^q)$ for every $q$th root of unity
$\zeta$. Since $A(t^q)\mid \Delta_K(t^p)$, the numbers
\[
        (\zeta\lambda)^p=\zeta^p\lambda^p
\]
are primitive $rs$th roots of unity for every $q$th root of unity $\zeta$.
Because $p$ and $q$ are relatively prime, the numbers $\zeta^p$, as $\zeta$
ranges over all $q$th roots of unity, also run over all $q$th roots of unity.
Hence, for every $q$th root of unity $\eta$, the product $\eta\lambda^p$ is a
primitive $rs$th root of unity.

Choose a prime $\ell$ dividing $q$. If $\ell$ is different from both
$r$ and $s$, then multiplying $\lambda^p$ by a primitive $\ell$th
root of unity $\eta$, which is also a $q$th root of unity, gives a root
$\eta\lambda^p$ of order $rs\ell$, contradiction.

It remains to consider the case $\ell=r$ or $\ell=s$. Suppose, for
example, that $\ell=r$; the other case is identical. Write
$\lambda^p=e^{2\pi i a/(rs)}$, where $\gcd(a,rs)=1$. Since $s$ is
invertible modulo $r$, choose $j$ so that $a+js\equiv 0\pmod r$. Then set
\[
        \eta=e^{2\pi i j/r}=e^{2\pi i j/\ell}.
\]
Since $\ell\mid q$, the number $\eta$ is a $q$th root of unity. However,
\[
        \eta\lambda^p=e^{2\pi i(a+js)/(rs)}
\]
is not a primitive $rs$th root of unity, since $r\mid a+js$. This is the desired contradiction,
and hence \Cref{cor:pqcable} applies.
\end{proof}

\begin{remark}\label{rmk:moreexamples}
The same argument readily applies to a connected sum of knots of the form
$T_{r,s}\#-T_{r,s}$, where $r$ and $s$ are distinct primes for each summand
and may vary from summand to summand, with no further restrictions.
\end{remark}

\section{Fixed-companion rigidity and genus bounds}\label{sec:fixed-cables-genus}

In this section, we prove the remaining results stated in the introduction. The proofs combine the ribbon-concordance obstruction from knot Floer homology with cabling formulae for immersed curves and concordance invariants. We begin with Theorem~\ref{thm:ribboncable}, which we recall.

\begin{reptheorem}{thm:ribboncable}
Let $K$ be a knot, and let $p>1$. Suppose that $J$ is a cable of $K$ with
\[
        J\leq K_{p,q}.
\]
Then $J$ is the $(p,q)$-cable of $K$.
\end{reptheorem}
\begin{proof}
Suppose that
\[
        K_{p',q'}\leq K_{p,q}
\]
where $p,p'>1$. Then Proposition~\ref{prop:cable-parameters} implies that either $|q|=|q'|=1$, or
\begin{equation*}\label{eq:pandq}
        \operatorname{sgn}(q)=\operatorname{sgn}(q')
        \qquad\text{and}\qquad
        \{p,|q|\}=\{p',|q'|\}.
\end{equation*}
In particular, it is not possible for exactly one of $|q|$ and $|q'|$ to be equal to $1$. If $K$ is the unknot, then the two cables are isotopic torus knots; hence, we may assume that $K$ is nontrivial.

First suppose that $|q'|,|q|>1$. If $p=p'$, then $q=q'$. Therefore, we may assume that $p\neq p'$, so we are in the swapped case. Thus suppose  that
\[
        p=|q'|
        \qquad\text{ and }\qquad
        |q|=p'.
\]

For the case $\mh(K)\neq 0$, if $p'<p$, then Proposition~\ref{prop:minheightcable} gives
\[
        \mh(K_{p',q'})=p'\cdot(\mh(K)-1)+1
        \,<\,
        p\cdot (\mh(K)-1)+1=\mh(K_{p,q}).
\]
This contradicts Proposition~\ref{prop:minheightobstruction}. If $p'>p$, then
\[
        g(K_{p',q'})-g(K_{p,q})
        =
        (p'-p)g(K)>0,
\]
contradicting the fact that genus is non-decreasing under ribbon
concordance~\cite[Theorem 1.5]{Zemke-ribbon}.

It remains to consider the case $\mh(K)=0$, that is, $\gamma(K)=\gamma_0(K)$. Since $K$ is
nontrivial, $h_0(K)>0$. By Proposition~\ref{prop:essentialheightcable}, we have
\[
        h_0(K_{p',q'})
        =
        p' \cdot h_0(K)+(p'-1)(p-1) \,\neq\,
        p \cdot h_0(K)+(p-1)(p'-1)=h_0(K_{p,q}).
\]
This contradicts
Lemma~\ref{lem:essential-ribbon}. Therefore the swapped case cannot occur.

Next, we consider the case when $|q'|=|q|=1$. If $p'\neq p$, then genus monotonicity again rules out the case $p'>p$. If $p'<p$ and $\mh(K)\neq 0$, then
Proposition~\ref{prop:minheightcable} gives
\[
        \mh(K_{p',q'})=p'\cdot(\mh(K)-1)+1
        \,<\,
        p\cdot(\mh(K)-1)+1=\mh(K_{p,q}),
\]
contradicting Proposition~\ref{prop:minheightobstruction}. If $p'<p$ and
$\mh(K)=0$, then $\gamma(K)=\gamma_0(K)$ and $h_0(K)>0$, while
Proposition~\ref{prop:essentialheightcable} gives
\[
        h_0(K_{p',q'})=p'\cdot h_0(K)
        \,\neq\,
        p\cdot h_0(K)=h_0(K_{p,q}).
\]
This contradicts Lemma~\ref{lem:essential-ribbon}. Thus $p'=p$.

We are left with $K_{p,1}$ and $K_{p,-1}$. Let $g=g(K)>0$, and write
\[
        V_m=\HFKhat_m(K,g).
\]
By genus detection~\cite{OS-genus}, the finite graded vector space
$V=\bigoplus_m V_m$ is nonzero. Proposition~\ref{prop:cablingHFKhat} gives
\[
        \HFKhat_m(K_{p,1},pg)\cong V_m .
\]
Using
\[
        K_{p,-1}\cong -\bigl((-K)_{p,1}\bigr),
\]
together with
Propositions~\ref{prop:symmetries} and~\ref{prop:cablingHFKhat}, gives
\[
        \HFKhat_m(K_{p,-1},pg)
        \cong
        V_{m-2(p-1)g}.
\]
Thus, in Alexander grading $pg$, the top summand of $K_{p,-1}$ is the top
summand of $K_{p,1}$ shifted by the nonzero Maslov shift $2(p-1)g$.

If there were a ribbon concordance in either direction between $K_{p,1}$ and
$K_{p,-1}$, then Zemke's ribbon-concordance injection on knot Floer homology
\cite[Theorem~1.1]{Zemke-ribbon} would give a grading-preserving injection
between these two finite graded vector spaces. This is impossible because
$V$ is finite-dimensional and nonzero. This rules out both directions and proves
that $K_{p',q'}$ and $K_{p,q}$ are isotopic.

Finally, suppose that $K$ is ribbon concordant in either direction to a proper
cable $K_{p,q}$. Then Proposition~\ref{prop:not-concordant-q-large} implies
that $|q|=1$. For the case $|q|=1$ and $K$ nontrivial, the fact that genus is non-decreasing under ribbon concordance implies that $K \ngeq K_{p,\pm1}$. It remains to show $K\nleq K_{p,\pm1}$.

Suppose first that $\mh(K)\neq0$, which implies that $\mh(K)\geq2$. Then
Proposition~\ref{prop:minheightcable} gives
\[
        \mh(K_{p,\pm1})=p\cdot\mh(K)-p+1.
\]
Therefore Proposition~\ref{prop:minheightobstruction} gives the desired
contradiction.

If $\mh(K)=0$, then $\gamma(K)=\gamma_0(K)$, and since $K$ is
nontrivial, $h_0(K)>0$. Lemma~\ref{lem:essential-ribbon} and
Proposition~\ref{prop:essentialheightcable} give
\[
        h_0(K)
        =
        h_0(K_{p,\pm1})
        =
        p\cdot h_0(K),
\]
which is a contradiction.
\end{proof}

We now prove Theorem~\ref{thm:genus}, whose statement we recall.

\begin{reptheorem}{thm:genus}
Let $K$ be a knot, and let $p>1$. Suppose that $J$ is a nontrivial knot
with
\[
        J\leq K_{p,q}.
\]
Then
\[
        g(J)\geq
        \left\lceil\frac{p\cdot h_0(K)}{2}\right\rceil+g(T_{p,q}).
\]
In particular, if $h_0(K)\neq 0$, then
$g(J)\geq p+g(T_{p,q})$. Moreover, if $J$ is a torus knot, then $J=T_{p,q}$.
\end{reptheorem}

\begin{proof}
By Lemma~\ref{lem:essential-ribbon} and
Proposition~\ref{prop:essentialheightcable},
\[
        h_0(J)=h_0(K_{p,q})
        =
        p\cdot h_0(K)+(p-1)(|q|-1)
        =
        p\cdot h_0(K)+2g(T_{p,q}).
\]
Since $h_0(J)\leq2g(J)$, it follows that
\[
        g(J)\geq
        \left\lceil\frac{p\cdot h_0(K)}{2}\right\rceil +g(T_{p,q}).
\]

Now suppose that $J$ is a torus knot. Write $J=T_{r,s}$, where
$r>1$ and $|s|>1$. Since a ribbon concordance is, in particular, a concordance, Proposition~\ref{prop:slice-full-LT} shows that
\[
        \operatorname{sgn}(q)=\operatorname{sgn}(s)
        \qquad\text{and}\qquad
        \{p,|q|\}=\{r,|s|\}.
\]
Therefore, $J=T_{r,s}$ is isotopic to $T_{p,q}$.
\end{proof}

Finally, we prove Theorem~\ref{thm:genus1}, whose statement we recall.

\begin{reptheorem}{thm:genus1}
Let $K$ be a knot, let $J$ be a genus-one knot, and let $p>1$.  Suppose that
\[
        J\leq K_{p,q}.
\]
Then
\[
        J=T_{p,q}
        \qquad\text{ and }\qquad
        \{p,|q|\}=\{2,3\}.
\]
In particular, $J$ is a trefoil knot, right-handed if $q>0$ and
left-handed if $q<0$.
\end{reptheorem}

\begin{proof}
If $\mh(J)\neq 0$, then Zemke's injectivity result~\cite[Theorem~1.7]{Zemke-ribbon},
together with the cabling formula for immersed curves~\cite[Theorem~1]{HW-cables},
implies that $\mh(K)\neq 0$, and hence $\mh(K)\geq2$. Propositions
\ref{prop:minheightobstruction} and \ref{prop:minheightcable} give
\[
        \mh(J)\geq\mh(K_{p,q})
        =p\cdot\mh(K)-p+1
        \geq p+1.
\]
Lemma \ref{lem:mhgenus} therefore implies that
\[
        g(J)\geq\left\lceil\frac{p+1}{2}\right\rceil,
\]
contradicting the assumption that $g(J)=1$.

Hence $\mh(J)=0$, and so $\gamma(J)=\gamma_0(J)$. Furthermore, $g(J)=1$
implies that $h_0(J)=2$, and then Lemma \ref{lem:essential-ribbon} implies
that $h_0(K_{p,q})=2$. Since $h_0(K)$ is even, Proposition
\ref{prop:essentialheightcable} implies that
\[
        h_0(K)=0
        \qquad\text{ and }\qquad
        (p-1)(|q|-1)=2.
\]
Thus $\{p,|q|\}=\{2,3\}$. Since $h_0(K)=0$, the cabling formula gives
\[
        \gamma_0(K_{p,q})=\gamma(T_{p,q}).
\]
The concordance invariance of $\gamma_0$ therefore implies that
\[
        \gamma(J)=\gamma_0(J)=\gamma(T_{p,q}).
\]
Finally, since knot Floer homology detects the trefoil~\cite{Ghiggini}, we conclude that $J=T_{p,q}$.
\end{proof}

\appendix

\section{Cable parameters under ordinary concordance}\label{app:examples}
We record here some folklore theorems on concordance of cables that follow from the classical Levine-Tristram signature~\cite{Levine,Tristram}. While these results may be well known to experts, we have not been able to find a concrete reference, so we record them here for the reader's convenience. As before, throughout this appendix, $(p,q)$-cabling is understood to have longitudinal winding number $p$ and to satisfy $\gcd(p,q)=1$.

We use the following convention for Levine--Tristram signatures. If $K$ is an oriented knot with Seifert matrix $V$, then for $\omega\in S^1\smallsetminus\{1\}$ we write $\sigma_K(\omega)$ for the signature of the Hermitian matrix
\[
        (1-\omega)V+(1-\overline{\omega})V^T .
\]
This convention is fixed so that the right-handed trefoil $T_{2,3}$ has
ordinary signature
\[
        \sigma_{T_{2,3}}(-1)=-2.
\]

\begin{proposition}\label{prop:not-concordant-q-large}
Let $K$ be a knot, and let $p>1$. If $K$ is concordant to $K_{p,q}$, then $|q|=1$.
\end{proposition}

\begin{proof}
Suppose that $K$ is concordant to $K_{p,q}$. For a knot $J$, set
\[
        I(J)=\int_0^1 \sigma_J(e^{2\pi i x})\,dx.
\]
Since $I$ is a concordance invariant, we have
\[
        I(K_{p,q})=I(K).
\]
Litherland's cabling formula for
Levine--Tristram signatures~\cite[Theorem~2]{Litherland-signatures} gives,
away from finitely many points,
\[
        \sigma_{K_{p,q}}(\omega)
        =
        \sigma_K(\omega^p)+\sigma_{T_{p,q}}(\omega).
\]
Since 
\[
        \int_0^1 \sigma_K(e^{2\pi i px})\,dx = I(K),
\]
integrating over the circle gives
\[
        I(K_{p,q})=I(K)+I(T_{p,q}).
\]
It follows that
\[
        I(T_{p,q})=0.
\]
Borodzik--Oleszkiewicz compute
\cite[Proposition~2.1]{BorodzikOleszkiewicz-signatures}
\[
        I(T_{p,q})
        =
        -\frac{\operatorname{sgn}(q)}{3}
        \left(p-\frac1p\right)
        \left(|q|-\frac{1}{|q|}\right).
\]
Since $p>1$, the equality $I(T_{p,q})=0$ forces $|q|=1$.
\end{proof}

\begin{proposition}\label{prop:cable-parameters}
Let $K$ be a knot, and let $p,p'>1$. If $K_{p,q}$ is concordant to
$K_{p',q'}$, then either $|q|=|q'|=1$, or
\[
        \operatorname{sgn}(q)=\operatorname{sgn}(q') \qquad \text{ and } \qquad 
        \{p,|q|\}=\{p',|q'|\}.
\]
\end{proposition}

\begin{proof}
Suppose that $K_{p,q}$ is concordant to $K_{p',q'}$. The same signature
integral used in Proposition~\ref{prop:not-concordant-q-large} gives
$I(K_{p,q})=I(K_{p',q'})$. Combining this with Litherland's cabling formula
gives
\[
        I(T_{p,q})=I(T_{p',q'}).
\]
If $|q|=1$, then $T_{p,q}$ is the unknot, so $I(T_{p,q})=0$. The above
equality gives $I(T_{p',q'})=0$, and hence $|q'|=1$. The same argument works
with $q$ and $q'$ interchanged.

We may therefore assume $|q|,|q'|>1$. Again, by \cite[Proposition~2.1]{BorodzikOleszkiewicz-signatures},
\[
        I(T_{p,q})
        =
        -\frac{\operatorname{sgn}(q)}{3}
        \left(p-\frac1p\right)
        \left(|q|-\frac1{|q|}\right),
\]
and the same formula holds for $I(T_{p',q'})$. Therefore, the equality
$I(T_{p,q})=I(T_{p',q'})$ implies
\[
        \operatorname{sgn}(q)=\operatorname{sgn}(q').
\]
Moreover, setting $r=|q|$ and $r'=|q'|$, we have
\begin{equation}\label{eq:LTintegral}
        \left(p-\frac1p\right)
        \left(r-\frac1r\right)
        =
        \left(p'-\frac1{p'}\right)
        \left(r'-\frac1{r'}\right).
\end{equation}

Next, instead of the average, for a knot $J$, we define
\[
        \widehat{I}(J)
        :=
        \int_0^1
        \sigma_J(e^{2\pi i x})e^{-2\pi i x}\,dx.
\]
We use the following elementary fact: if $f$ is a $1$-periodic integrable function and $n>1$ is an integer, then
\[
        \int_0^1 f(nx)e^{-2\pi i x}\,dx=0.
\]
Indeed, splitting the integral into the intervals $[j/n,(j+1)/n]$ and
putting $y=nx-j$ on the $j$th interval gives
\[
        \int_0^1 f(nx)e^{-2\pi i x}\,dx
        =
        \frac{1}{n}
        \left(\int_0^1 f(y)e^{-2\pi i y/n}\,dy\right)
        \sum_{j=0}^{n-1}e^{-2\pi i j/n}=0.
\]
As in the proof of Proposition~\ref{prop:not-concordant-q-large}, by applying
Litherland's cabling formula~\cite[Theorem~2]{Litherland-signatures}, we have
\begin{equation}\label{eq:LTsignature}
        \widehat{I}(K_{p,q})= \widehat{I}(T_{p,q}) \,
        =\,
        \widehat{I}(T_{p',q'}) =        \widehat{I}(K_{p',q'}) .
\end{equation}

In \cite[Section~3]{Borodzik-rho-iterated}, Borodzik computes that, for relatively prime integers $a,b>1$,
\[
        \widehat{I}(T_{a,b})
        =
        \frac{1}{\pi}\cot\left(\frac{\pi}{ab}\right).
\]
Therefore we have
\[
        \cot\left(\frac{\pi}{pr}\right)
        =
        \cot\left(\frac{\pi}{p'r'}\right),
\]
and hence
\[
        pr=p'r'.
\]
Combining this with \Cref{eq:LTintegral}, we obtain $p+r=p'+r'$, and therefore
\[
        \{p,r\}=\{p',r'\},
\]
which concludes the proof.
\end{proof}

\begin{remark}\label{rmk:KandJLTsignature}
Let $p,p'>1$. On the other hand, if we assume that $K_{p,q}$ is concordant to
$L_{p',q'}$, so that $K$ and $L$ are not necessarily the same knot, then we
only obtain $pq=p'q'$, which follows from \Cref{eq:LTsignature}. Moreover,
$|q|=1$ if and only if $|q'|=1$, since $\widehat{I}(T_{p,\pm 1})=0$.
\end{remark}

\begin{proposition}\label{prop:slice-full-LT}
Let $K$ be a slice knot, let $L$ be a knot, and let $p,p'>1$. If $K_{p,q}$ is concordant to $L_{p',q'}$, then either $|q|=|q'|=1$, or
\[
        \operatorname{sgn}(q)=\operatorname{sgn}(q')
        \qquad\text{and}\qquad
        \{p,|q|\}=\{p',|q'|\}.
\]
\end{proposition}

\begin{proof}
It is enough to prove the case $q>0$; the case $q<0$ is identical after
mirroring. By Remark~\ref{rmk:KandJLTsignature}, if $q=1$, then $|q'|=1$, so
we may assume that $q>1$. Moreover, we have $pq=p'q'$. Set $N=pq=p'q'$.

For a knot $J$ and a positive integer $m$, we generalize $\widehat{I}(J)$ by defining
\[
        \widehat{I}_m(J)
        :=
        \int_0^1
        \sigma_J(e^{2\pi i x})e^{-2\pi i mx}\,dx.
\]
Thus $\widehat{I}_1(J)=\widehat{I}(J)$. Set
\[
        F(x)
        =
        \sigma_{T_{p,q}}(e^{2\pi i x})
        -
        \sigma_{T_{p',q'}}(e^{2\pi i x}),
\] then by Litherland's cabling formula
\cite[Theorem~2]{Litherland-signatures}, we have
\[
        F(x)=\sigma_L(e^{2\pi i p'x})
\]
away from finitely many points. Note that $F$ is $1/p'$-periodic, therefore as in proof of \Cref{prop:cable-parameters} if $p'\nmid m$, then 
$$\widehat I_m(T_{p,q}) \,=\, \widehat I_m(T_{p',q'}).$$
For the computation, we again use Borodzik's formula~\cite[Section~3]{Borodzik-rho-iterated}. For relatively prime integers $a,b>1$ and an integer $m>0$ with $ab\nmid m$, taking the limit $\beta\to-2m$, with $\beta$ as in Borodzik's notation, gives
\begin{equation}\label{eq:ihatm}
\widehat{I}_m(T_{a,b})
=
\begin{cases}
\displaystyle
\frac{1}{\pi m}\cot\left(\frac{\pi m}{ab}\right),
& a\nmid m,\ b\nmid m, \\[1.2em]
\displaystyle
\frac{1}{\pi m}
\left\{
\cot\left(\frac{\pi m}{ab}\right)
-
a\cot\left(\frac{\pi m}{b}\right)
\right\},
& a\mid m,\ b\nmid m, \\[1.2em]
\displaystyle
\frac{1}{\pi m}
\left\{
\cot\left(\frac{\pi m}{ab}\right)
-
b\cot\left(\frac{\pi m}{a}\right)
\right\},
& b\mid m,\ a\nmid m.
\end{cases}
\end{equation}
Indeed, after setting $t=\pi\beta/2$, the three cases follow from the standard first-order expansions of $\sin t$ and the cotangent factors at $t=-\pi m$.

Suppose, toward a contradiction, that $\{p,q\}\neq\{p',q'\}$. Since
$pq=p'q'$ and $\gcd(p,q)=1$, at least one of $p$ and $q$ is divisible
by neither $p'$ nor $q'$. After interchanging $p$ and $q$ if necessary,
assume that
\begin{equation}\label{eq:p'q'notdividep}
    p'\nmid p
        \qquad\text{and}\qquad
        q'\nmid p.
\end{equation}
In particular, $p'\nmid p$, so the periodicity argument above gives
\[
        \widehat I_p(T_{p,q})
        =
        \widehat I_p(T_{p',q'}).
\]
On the other hand, Equation~\eqref{eq:ihatm}, together with
$N=pq=p'q'$, gives
\[
\widehat I_p(T_{p,q})
=
\frac{1}{\pi p}
\left\{
\cot\left(\frac{\pi p}{N}\right)
-p\cot\left(\frac{\pi p}{q}\right)
\right\} \qquad \text{ and } \qquad
\widehat I_p(T_{p',q'})
=
\frac{1}{\pi p}
\cot\left(\frac{\pi p}{N}\right).
\]
Consequently,
\[
        \cot\left(\frac{\pi p}{q}\right)=0.
\]

This can occur only when $p/q$ is a half-integer. It follows that $q=2$, and hence $p$ is odd. In this case, we instead consider $\widehat I_2$. Neither $p'$ nor $q'$ equals $2$, since otherwise we would contradict the assumption in \eqref{eq:p'q'notdividep}. In particular, $p'\nmid2$, so the periodicity argument gives
\[
        \widehat I_2(T_{p,2})
        =
        \widehat I_2(T_{p',q'}).
\]
Equation~\eqref{eq:ihatm}, together with $N=2p=p'q'$, gives
\[
\widehat I_2(T_{p,2})
=
\frac{1}{2\pi}
\left\{
\cot\left(\frac{\pi}{p}\right)
-2\cot\left(\frac{2\pi}{p}\right)
\right\}
\qquad \text{ and } \qquad 
\widehat I_2(T_{p',q'})
=
\frac{1}{2\pi}\cot\left(\frac{\pi}{p}\right).
\]
Consequently,
\[
        \cot\left(\frac{2\pi}{p}\right)=0,
\]
which is impossible because $p>1$ is odd. This proves the proposition.
\end{proof}

\bibliographystyle{alpha}
\bibliography{bib}

\end{document}